\documentclass[12pt,leqno, a4paper]{amsart}
\usepackage{amsmath,amssymb,amsfonts}
\usepackage{graphicx,xcolor}
\usepackage{mathrsfs}
\usepackage{bbm}
\usepackage[margin=0.9 in]{geometry}
\usepackage{mathrsfs}
\newtheorem{theorem}{Theorem}[section]
\newtheorem{lem}{Lemma}[section]

\newtheorem{remark}{Remark}[section]

\newcommand{\Z}{\mathbb{Z}}
\newcommand{\R}{\mathbb{R}}
\newcommand{\C}{\mathbb{C}}
\newcommand{\F}{\mathscr{F}}

\usepackage{mathrsfs}

\newtheorem{MTheo}{Theorem}

\begin{document}
\title[Global solution of 2D Dirac equation]{Global $H^s, s>0$ Large Data Solutions of 2D Dirac Equation with Hartree Type Interaction}
\author{Vladimir Georgiev  and Boris Shakarov }
\address{V. Georgiev
\newline Dipartimento di Matematica Universit\`a di Pisa
Largo B. Pontecorvo 5, 56100 Pisa, Italy\\
 and \\
 Faculty of Science and Engineering \\ Waseda University \\
 3-4-1, Okubo, Shinjuku-ku, Tokyo 169-8555 \\
Japan and IMI--BAS, Acad.
Georgi Bonchev Str., Block 8, 1113 Sofia, Bulgaria}
\email{georgiev@dm.unipi.it}
\address{B.Shakarov, Gran Sasso Science Institute, Via Michele Iacobucci, 2, 67100 L'Aquila, Italy}
\email{shabor993@gmail.com}

\begin{abstract}
Local and global well - posedness of the solution to the two space dimensional Dirac equation with Hartree type nonlinearity is established with the initial datum in the space $H^s(\R^2, \C^2)$ with $s >0.$.
\end{abstract}

\maketitle

\section{Introduction}

The 2D Dirac type models have manifested an increasing role in the last years in connection with the new hypothetical field called anionic Dirac matter that is
 a type of quasiparticle that can only occur in two - dimensional systems (see for example \cite{WBB}).

  The interaction between fermions (a typical example is the electron) confined to a plane with massive bosons particles can be interpreted  as a 2D nonlinear Dirac with cubic nonlinearity of Hartree type (see \cite{chad}, \cite{Ae}). A General overview on the application of variational methods for this type of models can be found in \cite{ELS}.

In this paper we study
the 2D Dirac equation with a non - local interaction term of Hartree  type. To be more precise, we plan to study local and global well - posedness of the Cauchy problem
\begin{equation}\label{eq:base}
    \begin{cases}
    \begin{aligned}
        & -\mathrm{i}  \gamma^\mu \partial_\mu \psi + m \psi = ((b-\Delta)^{-1}|\psi|^{2}) \psi, \\
 &  \psi(0,x) = \psi_0 \in H^s(\mathbb{R}^2),
    \end{aligned}
    \end{cases}
\end{equation}
where $m>0$ is the mass of the spinor, $ s > 0$, $b > 0$, $\psi: \R^{1 + 2} \rightarrow \C^2 $, the Dirac matrices $\gamma^\mu$ are given by
\begin{equation*}
\gamma^0 = \begin{pmatrix} 1 & 0\\ 0 & -1 \end{pmatrix}, \hspace{1cm} \gamma^1 = \begin{pmatrix} 0 & i\\ i & 0 \end{pmatrix}, \hspace{1cm} \gamma^2 = \begin{pmatrix} 0 & 1\\ -1 & 0 \end{pmatrix}
\end{equation*}
and $\langle \cdot , \cdot \rangle$ is the inner product in $\mathbb{C}^2$.
We use the standard summation rule for repeated indices as well the classical rule of raising and lowering indices by the aid of the metric $\mathrm{diag}(1,-1-1)$  so that $$\partial^0=\partial_0= \partial_t, -\partial^j=\partial_j = \partial_{x_j}, j = 1,2.$$
This Cauchy problem was treated in \cite{T} and the main result of that work is the global existence and scattering with an initial datum
which has sufficiently small Sobolev norm $H^s(\R^2)$ with $s >0$.

Our main goal is to improve this result and establish local and global well - posedness with initial data in $H^s(\R^2)$ without the smallness assumption.

For the case of the local interaction of the type $( \langle \gamma^0 \psi,  \psi \rangle) \psi$ one can see the corresponding results in \cite{herr} where the initial datum is in $H^{\frac{1}{2}}.$

In order to state our main result, we can rewrite the Cauchy problem \eqref{eq:base} in the equivalent form
\begin{equation} \label{eq:base1}
    \begin{cases}
    \begin{aligned}
  &   \mathrm{i}   \partial_t \psi =  \mathcal{D}_m \psi - ((b-\Delta)^{-1} |\psi|^2) \gamma^0\psi, \\
 &  \psi(0,x) = f \in H^s(\mathbb{R}^2),
    \end{aligned}
    \end{cases}
\end{equation}{}
where
\begin{equation*}
    \mathcal{D}_m = - \mathrm{i} \alpha^j \partial_j + m \gamma^0
\end{equation*}
is a self-adjoint operator with the domain $H^1(\mathrm{R}^2)$  due to the fact that the matrices
$$ \alpha^1 = \gamma^0 \gamma^1 =\begin{pmatrix} 0 & i\\ -i & 0 \end{pmatrix} \hspace{0.5cm} \alpha^2 = \gamma^0 \gamma^2 = \begin{pmatrix} 0 & 1\\ 1 & 0 \end{pmatrix},$$
and $\gamma^0$ are self - adjoint. Thus, we look for a solution $\psi \in C([0,\infty);H^s(\R^2))$ satisfying the integral equation
\begin{equation}\label{eq:base2}
  \psi(t) = e^{-\mathrm{i}t\mathcal{D}_m} \psi_0  + \mathrm{i}\int_0^t e^{-\mathrm{i}(t-s)\mathcal{D}_m} ((b - \Delta)^{-1} |\psi|^2) \gamma^0\psi(s) ds,
\end{equation}
\begin{MTheo} \label{TMain} [Global existence]
 For any $s >0$ and any $\psi_0(x) \in H^s(\mathbb{R}^2) $ there exists a unique solution $\psi(t,x) \in C([0,\infty);H^s(\R^2))$ to integral equation \eqref{eq:base2}.
 Moreover, for any $t \in [0,\infty)$,
\begin{equation} \| \psi(t,\cdot) \|_{L^2 (\R^2)} = \|\psi_0\|_{L^2(\R^2)}.
\end{equation}
\end{MTheo}
The key point in the proof is the following Brezis-Gallou\"et Type Inequality (see \cite{BrGa}):
\begin{MTheo}  \label{T1} For any $b > 0$ and $s > 0$ there exists a constant $C=C(b,s)>0$ such that, for any $f \in B^s_{1,\infty}(\mathbb{R}^2) $ the following inequality is true
  \begin{equation} \label{eq:result1}
  \begin{split}
\| ( b - \Delta)^{-1} f \|_{L^\infty(\R^2)}   + \| ( b - \Delta) ^{-1} (1-\Delta)^{s/2}f   \|_{L^\frac{2}{s}(\R^2)} \leq C \|f\|_{L^1(\R^2)} \log \left( 2 + \frac{\|f\|_{B^s_{1,\infty}(\R^2)}}{\|f\|_{L^1(\R^2)}}\right) .
\end{split}
\end{equation}
\end{MTheo}

Concerning the behavior of energy type norm $\|\psi (t)\|_{H^\frac{1}{2}(\R^2)}$ we can show double exponential growth for the Dirac case, namely we have the following.
\begin{MTheo} \label{TMain1} [Growth of $H^{1/2}$ norm]
 For any $\psi_0(x) \in H^{1/2}(\mathbb{R}^2) $ there exist $C_1(\| \psi_0\|_{H^s}) >0$ and $C_2>0$ such that the solution $\psi(t,x) \in C([0,\infty);H^{1/2}(\R^2))$ to integral equation \eqref{eq:base2} satisfies the estimate
\begin{equation} \label{supexp estm}
    \|\psi (t)\|_{H^\frac{1}{2}(\R^2)} \leq e^{C_1 e^{C_2 t}}.
\end{equation}
\end{MTheo}

Since the kinetic energy of the Dirac equation is determined by the indefinite form
$$ \langle \mathcal{D}_m \psi(t), \psi(t) \rangle_{L^2}$$
we can obtain the following better exponential bound.
\begin{MTheo} \label{TMain2} [Growth of the kinetic energy]
 For any $\psi_0 (x) \in H^{1/2}(\mathbb{R}^2) $, there exists a constant $C= C( \| \psi_0\|_{H^\frac{1}{2}}) >0$ such that the solution $\psi(t,x) \in C([0,\infty);H^{1/2}(\R^2))$ to integral equation \eqref{eq:base2} satisfies the estimate
\begin{equation} \label{supexp estmm}
 \left|   \langle \mathcal{D}_m \psi(t), \psi(t) \rangle_{L^2} \right| \lesssim e^{Ct}.
\end{equation}
\end{MTheo}

We will use the following notations:
\begin{itemize}
\item $ f(t) \lesssim g(t)$ means that there exists a constant $C>0$, possibly depending on some fixed values but independent of $t$, such that $f(t) \leq Cg(t)$;
\item $H^s := H^s(\R^2,\C^2)$ with $s > 0$. In the same way $L^p:= L^p(\R^2,\C^2)$;
\item Throughout the equations, we use $C$ for positive constants coming from various known inequalities. With abuse of notation, $C$ can change.
\item We define the Fourier transform and the Anti-Fourier transform in the usual way: $\hat{\psi}(\xi) = \int e^{-i x \cdot \xi} \psi(x) dx$ and $\F^{-1}(\psi(x)) = \frac{1}{4\pi^2}\int e^{i x \cdot \xi} \psi(\xi) d\xi$.
\item Let $\hat{\rho}(\xi) $ be a radial, positive, Schwartz function, equal to $1$ in $1 \leq |\xi| \leq 2$, to $0$ for $|\xi|\leq 1 - \frac{1}{7}$ and $|\xi|\geq 2 + \frac{1}{7}$ and such that $\sum_{j \in \Z} \hat{\rho}( 2^{-j} \xi) = 1$. We define $\hat{\rho}_{(0)}(\xi) = \sum_{j \leq 0} \hat{\rho}( 2^{-j} \xi)$ Let $\psi$ be a Schwartz function. We define $\psi_{(j)} = \F^{-1} ( \hat{\rho}( 2^{-j} \xi) \hat{\psi}(\xi))$ and  $\psi_{(0)} = \F^{-1} ( \hat{\rho}_{(0)}(\xi) \hat{\psi}(\xi))$
 The Besov space $B^s_{p,q}$ is defined as the semi-normed space of functions such that
\begin{equation*}
\| \psi \|_{B^s_{p,q}} = \| \psi_{(0)} \|_{L^p} + \left( \sum_{j=1}^\infty (2^{js} \|\psi_{(j)}\|_{L^p})^q \right)^\frac{1}{q} < \infty,
\end{equation*}
following substantially the definition found in \cite{Graf}.
\end{itemize}
\section{Local and Global Existence}

\begin{theorem} [Local Existence] \label{local existence} For any $s > 0$ and any $\psi_0 \in H^s(\R^2)$, there exists a time $T=T(\| \psi_0\|_{H^s})>0$ and a unique solution $\psi \in C([0,T);H^s(\R^2))$ to the equation \eqref{eq:base2}.
\end{theorem}
 \begin{proof}
Let $R:= \| \psi_0\|_{H^s}$ and take $b = 1$. We use a contraction principle in the space $X_T = \{ \psi \in  L^\infty([0,T];H^s(\R^2));  \| \psi\|_{L^\infty((0,T),H^s(\R^2))}   \leq 2R \}$ equipped with the distance
\begin{equation*}
    d(u,v) = \| u - v\|_{L^\infty(0,T), H^s(\R^2)},
\end{equation*}{}
with $T = T(R)$ to be chosen later. Observe that $(X_T, d)$ is a complete metric space. We define the map
\begin{equation}\label{the contraction}
S( \psi) = e^{-it D_m}\psi_0 + i \int_0^t e^{-i (t-\tau) D_m}  ((1 - \Delta)^{-1} |\psi (\tau)|^2)\gamma_0 \psi(\tau) d\tau,
\end{equation}
and we are going to prove that $S$ is a contraction in the space $X_T$. Indeed, taking an element $\psi \in X_T$, one can see that
\begin{equation} \label{eq,b}
\begin{split}
\| S(\psi) (t) \|_{H^s(\R^2)} & \leq \| \psi_0 \|_{H^s(\R^2)}   + \int_0^t \| (1-\Delta)^{s/2} \left[ ((1 - \Delta)^{-1} |\psi(\tau)|^2) \gamma_0 \psi(\tau) \right] \|_{L^2(\R^2)} d\tau \\ &
\leq  \| \psi_0 \|_{H^s(\R^2)} +  C \int_0^t \left\| (1-\Delta)^{-1+s/2}   |\psi(\tau)|^2)  \right\|_{L^\frac{2}{s}} \left\| \psi(\tau) \right\|_{L^\frac{2}{1-s}} \\& + \left\|  ((1 - \Delta)^{-1} |\psi(\tau)|^2)  \right\|_{L^\infty} \| (1-\Delta)^{s/2}\psi(\tau) \|_{L^2} d\tau \\&
\leq \| \psi_0\|_{H^s} + C \int_0^t \| \psi(\tau) \|_{L^2}^2 \log \left(2 + \frac{ \| \psi(\tau)\|^2_{H^s}}{\| \psi(\tau) \|^2_{L^2}}  \right)  \| \psi(\tau)\|_{H^s} d\tau,
\end{split}
\end{equation}
where we used \eqref{eq:result1}, the Sobolev's embedding $\| \psi \|_{L^{\frac{2}{1-s}}(\R^2)} \leq C \| \psi\|_{H^s(\R^2)}$ and that $\| |\psi|^2 \|_{B^s_{1,\infty}} \leq C \| \psi \|_{B^s_{2,2}}^2$.
Indeed one has that, for any $\phi, \psi \in H^s$,
 \begin{equation} \label{Bs norm}
 \begin{split}
 \sup_{j \geq 0} 2^{js} \| (\phi \psi)_{(j)} \|_{L^1} & \lesssim \sup_{j \geq 0}2^{js}  \sum_{m \leq j} \left\|  \phi_{(m)} \psi_{(j-m)} \right\|_{L^1} \leq \sup_{j \geq 0} \sum_{m \leq j} 2^{ms} \|  \phi_{(m)} \|_{L^2}  2^{(j-m)s}\|\psi_{(j-m)}\|_{L^2} \\ &
 \leq \| \phi\|_{B^s_{2,2}} \|\psi\|_{B^s_{2,2}}
 \end{split}
\end{equation}
 It follows that
\begin{equation} \label{eq stima 2}
\begin{split}
\| S(\psi) \|_{L^\infty(0,T);H^s}  &\leq R +
C T  \| \psi(\tau) \|_{L^\infty(0,T),L^2}^2 \| \psi \|_{L^\infty(0,T),H^s} \log \left(2 + \sup_{\tau \in (0,T)} \left(\frac{ \| \psi (\tau)\|^2_{H^s}}{\| \psi(\tau) \|^2_{L^2}} \right)  \right) \\ &
\leq R + C_{R} T,
\end{split}
\end{equation}
where $C_R = C \| \psi(\tau) \|_{L^\infty(0,T),L^2}^2 \| \psi \|_{L^\infty(0,T),H^s} \log \left(2 + \sup_{\tau \in (0,T)} \left(\frac{ \| \psi (\tau)\|^2_{H^s}}{\| \psi(\tau) \|^2_{L^2}} \right)  \right)$ which is finite. Indeed $\psi \in X_T$ implies that $\| \psi(\tau) \|_{L^\infty(0,T),L^2} \in [0,2R]$, and $\lim_{M \rightarrow 0} M \log (2 + \frac{1}{M}) = 0$. A posteriori, once the local existence is established and the conservation of the $L^2$ norm is proved, we have that $C_R \leq C \| \psi_0\|_{L^2}^2 R \log(2 + \frac{4R^2}{\| \psi_0\|_{L^2}^2})$. In any case, we choose $T \leq \frac{R}{C_R}$ so that $S:X_T \rightarrow X_T$.

Now we show that \eqref{the contraction} is a contraction in $X_T$. Note that for any $\psi, \phi \in X_T$, one has
\begin{equation*}
\begin{split}
 \| \psi - \phi & \|_{L^\infty(0,T),H^s}   = \\ & \left\| \int_0^t e^{-i (t-\tau) D_m}  ((1 - \Delta)^{-1} |\psi (\tau)|^2)\gamma_0 \psi(\tau) - ((1 - \Delta)^{-1} |\phi (\tau)|^2)\gamma_0 \phi(\tau)  d\tau \right\|_{L^\infty(0,T),H^s} \\ &
 \leq C T \left( \| ((1 - \Delta)^{-1} ( |\psi|^2 + |\phi|^2))(\psi - \phi) \|_{L^\infty(0,T),H^s}  \right) + \\ & +  C T \left( \| ((1 - \Delta)^{-1} ( |\psi|^2 - |\phi|^2))(\psi + \phi) \|_{L^\infty(0,T),H^s} \right).
        \end{split}
\end{equation*}
From the estimate \eqref{eq:result1} it follows that
\begin{equation*}
\begin{split}
\| ((1 - & \Delta)^{-1} ( |\psi|^2 + |\phi|^2))(\psi - \phi) \|_{L^\infty(0,T),H^s} \\ &
\leq C \| \psi - \phi \|_{L^\infty(0,T),H^s} \sup_{\tau \in (0,T)} \left( \| |\psi|^2 + |\phi|^2 \|_{L^1}  \log \left (2 +  \frac{\| |\psi|^2 + |\phi|^2 \|_{B^s_{1,\infty}}}{ \| |\psi|^2 + |\phi|^2 \|_{L^1}} \right)\right) \\ &
\leq C_R \| \psi - \phi \|_{L^\infty(0,T),H^s}
    \end{split}{}
\end{equation*}
and
\begin{equation*}
\begin{split}
\|  ((1 - \Delta)^{-1}( & |\psi|^2 -  |\phi|^2)) (\psi + \phi) \|_{L^\infty(0,T),H^s} \\ &
\leq C \| \psi + \phi \|_{L^\infty(0,T), H^s} \sup_{t \in [0,T)} \left( \| |\psi|^2 - |\phi|^2 \|_{L^1} \log \left( 2 + \frac{\| |\psi|^2 - |\phi|^2 \|_{B^s_{1,\infty}}}{\| \psi|^2 - |\phi|^2 \|_{L^1}}   \right) \right) \\ &
\leq C R \| \psi - \phi \|_{L^\infty(0,T),H^s}  \sup_{t \in [0,T)} \left( \frac{ \| |\psi|^2 - |\phi|^2 \|_{L^1}}{\| \psi - \phi\|_{H^s} } \log \left( 2 + \frac{\| |\psi|^2 - |\phi|^2 \|_{B^s_{1,\infty}}}{\| \psi|^2 - |\phi|^2 \|_{L^1}}   \right) \right) \\ &
\leq CR \| \psi - \phi \|_{L^\infty(0,T),H^s} \sup_{t \in [0,T)} \left(  \frac{ \| |\psi|^2 - |\phi|^2 \|_{L^1}}{\| \psi - \phi\|_{H^s} } \log \left( 2 + 4R\frac{ \| \psi - \phi \|_{H^s}}{\| \psi|^2 - |\phi|^2 \|_{L^1}}\right) \right) \\ &
 \leq C_R \| \psi - \phi \|_{L^\infty(0,T),H^s},
    \end{split}
\end{equation*}
where we used \eqref{Bs norm} to have
\begin{equation*}
\| |\psi|^2 - |\phi|^2 \|_{B^s_{1,\infty}} \leq \| (\bar{\psi} - \bar{\phi}) \psi + \bar{\phi}(\psi - \phi)\|_{B^s_{1,\infty}} \leq 4R \|\psi - \phi\|_{H^s}
\end{equation*}
 and that $ \frac{ \| |\psi|^2 - |\phi|^2 \|_{L^1}}{\| \psi - \phi\|_{H^s} } \leq C_R$ since $\psi \neq \phi$. We conclude that \eqref{the contraction} is a contraction in $X_T$.
The conservation of the $L^2$ norm is shown in the appendix (see Lemma \ref{l:m1})
\end{proof}
\begin{theorem} [Global Existence] For any $s > 0$ and any $\psi_0 \in H^s(\R^2)$, there exists a unique solution $\psi \in C([0,+\infty);H^s(\R^2))$ to the equation \eqref{eq:base2}.
\end{theorem}
\begin{proof}
Fix a $s>0$ let $\psi_0 \in H^s$. From the integral form \eqref{eq:base2} and \eqref{eq,b}, one can easily derive that for any $t \in [0,T)$, where $T>0$ is given by \ref{local existence},
\begin{equation*}
\frac{d}{dt} \| \psi (t) \|_{H^s} \leq C \log \left(2 + \|\psi(t)\|_{H^s} \right) \| \psi(t) \|_{H^s}.
\end{equation*}
Since one has
\begin{equation*}
\frac{d}{dt}  \| \psi (t) \|_{H^s} \leq \begin{cases}
& C \log(4) \| \psi(t)\|_{H^s} \mbox{ if } \| \psi(t)\|_{H^s} \leq 2 \\ &
C \log (2\| \psi(t)\|_{H^s} ) \| \psi(t)\|_{H^s} \leq 2C \log(\| \psi\|_{H^s}) \| \psi (t) \|_{H^s} \mbox{ if } \| \psi(t)\|_{H^s} \geq 2
\end{cases}
\end{equation*}
it follows that
\begin{equation} \label{superexpo growth}
\| \psi (t) \|_{H^s} \leq \begin{cases}
& C_1 e^{C_2}t \mbox{ if } \| \psi(t)\|_{H^s} \leq 2 \\ &
e^{C_1 e^{C_2 t}} \mbox{ if } \| \psi(t)\|_{H^s} \geq 2
\end{cases}
\end{equation}
where $C_1$ depends on $\| \psi_0\|_{H^s}$ and $C_2$ depends on $C$.
\end{proof}
\section{Proof of the Theorem \ref{T1}}
In this section we prove
\begin{MTheo}  For any $b > 0$ and $s > 0$ there exists a constant $C=C(b,s)>0$ such that, for any $\psi \in B^s_{1,\infty}(\R^2) $ the following inequality is true
  \begin{equation}
  \begin{split}
\| ( b - \Delta)^{-1} \psi \|_{L^\infty}   + \| ( b - \Delta) ^{-1} (1-\Delta)^{s/2}\psi   \|_{L^\frac{2}{s}} \leq C \|\psi\|_{L^1(\R^2)} \ln \left( 2 + \frac{\|\psi\|_{B^s_{1,\infty}}}{\|\psi\|_{L^1}}\right) .
\end{split}
\end{equation}
\end{MTheo}
Recall the definition of $ \psi_{(j)}$ given in the introduction.
We will use the following lemmas (see e.g \cite{BL})
\begin{lem}
For any $j \in \Z$ and any $q,r$ such that $1 \leq q < r \leq \infty$, there exists a constant $C> 0$ such that
\begin{equation}\label{eq.B1}
    \| \psi_{(j)} \|_{L^r} \lesssim 2^{2j(\frac{1}{q} - \frac{1}{r})} \|\psi_{(j)}\|_{L^q}.
\end{equation}
\end{lem}
\begin{lem}
Assume $\psi$ to be a Schwartz function and assume $\psi_{(j)} \in L^p$ for some $p \in [1,\infty]$. Then, for any $b > 0$, $s \in \R$ and $j \geq 1$, there exists a constant $C>0$ such that
\begin{equation*}
\| (b - \Delta)^{\frac{s}{2}} \psi_{(j)} \|_{L^p} \leq C 2^{sj} \| \psi_{(j)} \|_{L^p}
\end{equation*}
\end{lem}
\begin{proof}
Let $\psi \in B^s_{1,\infty}(\R^2)$ and observe that $\psi \in L^1$. Let $M \in \R$ a constant such that  $s M = \ln \left(2 + \frac{\|\psi\|_{B^s_{1,\infty}(\R^2)}}{\|\psi\|_{L^1(\R^2)}}\right), $ and let $[M]$ be the integer part of $M$. Observe that $\left( (b - \Delta)^{-1} \psi \right)_{(j)} = (b - \Delta)^{-1} \psi_{(j)}$ and so
\begin{equation*}
\begin{split}
\| ( b - \Delta)^{-1} \psi & \|_{L^\infty} =  \left\| ( b - \Delta)^{-1} \left( \sum_{j \in \Z} \psi_{(j)}  \right) \right\|_{L^\infty} \lesssim \| ( b - \Delta)^{-1} \psi_{(0)}\|_{L^\infty} + \sum_{j > 0} \left\| ( b - \Delta)^{-1} \psi_{(j)} \right\|_{L^\infty}
\end{split}
\end{equation*}
From the two lemmas, one can obtain that
\begin{equation*}
\begin{split}
\sum_{j > 0} \| ( b - \Delta) ^{-1} & \psi_{(j)}\|_{L^\infty} \lesssim \sum_{j >0}  2^{-2j} \|\psi_{(j)}\|_{L^\infty}
 \lesssim \sum_{j > 0}  \|\psi_{(j)}\|_{L^1} = \sum_{0 \leq j \leq [M]}  \|\psi_{(j)}\|_{L^1} + \sum_{ j > [M] + 1}  \|\psi_{(j)}\|_{L^1} \\ &
 \lesssim [M] \| \psi \|_{L^1} + \sum_{ j \geq [M] + 1}  2^{-js}2^{js}\|\psi_{(j)}\|_{L^1}  \lesssim [M] \| \psi \|_{L^1}  + 2^{-s([M] + 1)}\| \psi \|_{B^s_{1,\infty}},
\end{split}
\end{equation*}
Moreover, since $\| \F^{-1}(\hat{\rho_{(0)}}) \|_{L^\infty} \leq C$, one has
\begin{equation*}
\| ( b - \Delta)^{-1} \psi_{(0)}\|_{L^\infty} \leq C \| \F^{-1}(\hat{\rho}_{(0)} )\|_{L^\infty} \| ( b - \Delta)^{-1} \psi \|_{L^1} \leq C \| \psi\|_{L^1}.
\end{equation*}
Moreover, if $b=1$, then
\begin{equation*}
\begin{split}
\| ( 1 - \Delta)^{-1 +\frac{s}{2}} \psi  \|_{L^\frac{s}{2}} \lesssim \| ( 1 - \Delta)^{-1 + \frac{s}{2}} \psi_{(0)}\|_{L^\frac{s}{2}} + \sum_{j > 0} \left\| ( 1 - \Delta)^{-1 + \frac{s}{2}} \psi_{(j)} \right\|_{L^\frac{s}{2}}
\end{split}
\end{equation*}
and using the lemmas before, one has
\begin{equation*}
\begin{split}
  \sum_{j > 0} \|  &(1-\Delta)^{-1 + s/2} \psi_{(j)}\|_{L^\frac{2}{s}}
 \lesssim \sum_{0 \leq j \leq [M]} 2^{(-2 + s)j}\| \psi_{(j)}\|_{L^\frac{2}{s}} +  \sum_{j > [M]}2^{-js} 2^{js} 2^{-2j+sj} \| \psi_{(j)}\|_{L^\frac{2}{s}} \\
& \lesssim \sum_{0 \leq j \leq [M]} 2^{(-2+s)j} 2^{(2 - s)j}\| \psi_{(j)}\|_{L^1} +  \sum_{j \geq [M]+1} 2^{-js}2^{js} \| \psi_{(j)}\|_{L^1}\\
& \lesssim [M] \| \psi\|_{L^1} + 2^{-s([M]+1)} \| \psi\|_{B^s_{1,\infty}},
\end{split}
\end{equation*}
and, as before
\begin{equation*}
\| ( 1 - \Delta)^{-1 + \frac{s}{2}} \psi_{(0)}\|_{L^\frac{s}{2}} \leq C \| \F^{-1}(\hat{\rho_{(0)}})\|_{L^\frac{2}{s}} \| \psi\|_{L^1} \leq C  \| \psi\|_{L^1}.
\end{equation*}
The result follows easily from the choice of $M$.
\end{proof}

\section{Exponential  Bound  of the Kinetic  Energy}
This section is dedicated to find a result for the growth rate of the kinetic energy of a solution to the equation \eqref{eq:base1} with initial data in $H^\frac{1}{2}$. Note that we have already proved a super exponential growth of the norm for any $s > 0$. In particular, if $\psi \in C([0, \infty); H^\frac{1}{2}(\R^2))$, then
\begin{equation} \label{supexp est}
    \|\psi (t)\|_{H^\frac{1}{2}(\R^2)} \lesssim e^{C_1e^{C_2t}}.
\end{equation}
As soon as $\psi(t) \in H^{\frac{1}{2}}$, we gain the total energy conservation, that is, for any $t \in [0,\infty)$,
\begin{equation*}
E(\psi_0) = E(\psi(t)) := \frac{1}{2} \langle D_m \psi(t), \psi(t) \rangle_{L^2(\R^2)} - \frac{1}{4} \| (1 - \Delta)^{-\frac{1}{2}} |\psi(t)|^2 \|^2_{L^2}.
\end{equation*}
So there exists a constant $C>0$ such that
\begin{equation}\label{En cons}
\langle D_m \psi(t), \psi(t) \rangle_{L^2} \leq C  \| (1 - \Delta)^{-\frac{1}{2}} |\psi(t)|^2 \|^2_{L^2}.
\end{equation}
\begin{lem}
For any $\varepsilon >0$, let $p = \frac{2 + 2\varepsilon}{1 + 3\varepsilon} \in (\frac{2}{3},2)$. Then the following inequality is true
\begin{equation}\label{pot est}
\| (1 - \Delta)^{-\frac{1}{2}} |\psi(t)|^2 \|_{L^2} \lesssim \left(\frac{2^\varepsilon}{\varepsilon}\right)^\frac{1}{p}  \| \psi(t)\|^\frac{4\varepsilon}{1 + \varepsilon}_{H^\frac{1}{2}}
\end{equation}
\end{lem}
\begin{proof}
Let $G$ be the Kernel of the operator $(1 - \Delta)^{-\frac{1}{2}}$. Then, from the Young inequality, for any $\varepsilon > 0$, it follows that
\begin{equation*}
\| (1 - \Delta)^{-\frac{1}{2}} |\psi(t)|^2 \|_{L^2} =   \| G \ast |\psi(t)|^2 \|_{L^2} \lesssim \| G \|_{L^p} \| |\psi(t)|^2 \|_{L^{1+\varepsilon}}= \| G \|_{L^p} \| \psi(t) \|_{L^{2+2\varepsilon}}^2,
\end{equation*}
where $ p = \frac{2 + 2\varepsilon}{1 + 3\varepsilon} < 2$. It is known (see e.g \cite{Graf}) that the kernel of the Bessel operator $(1 - \Delta)^{-\frac{1}{2}}$ can be estimated as
\begin{equation*}
    G(x) \leq C e^{-\frac{|x|}{2}}  \mathbbm{1}_{|x|\geq 2} + C |x|^{-1} \mathbbm{1}_{|x|\leq 2},
\end{equation*}
and so, for any $\varepsilon >0$, one has
\begin{equation*}
\begin{split}
 \| G(x) \|_{L^p}^p  & \lesssim \int_2^\infty e^{-\frac{r}{2} \frac{2 + 2\varepsilon}{1 + 3\varepsilon}} r dr + \int_0^2 r^{- \frac{2 + 2\varepsilon}{1 + 3\varepsilon} +1} dr \lesssim  \left( \frac{2 + 2\varepsilon}{1 + 3\varepsilon} + 1\right)\left( \frac{1 + \varepsilon}{1 + 3\varepsilon}\right)^{-2} e^{- \frac{2 + 2\varepsilon}{1 + 3\varepsilon}} + \frac{2^\varepsilon}{\varepsilon} \\ &
= \left( \frac{15 \varepsilon^2 + 14 \varepsilon +3}{(1 + \varepsilon)^2} \right) e^{- \frac{2 + 2\varepsilon}{1 + 3\varepsilon}} + \frac{2^\varepsilon}{\varepsilon} \lesssim \frac{2^\varepsilon}{\varepsilon},
  \end{split}
\end{equation*}
and by interpolation and the conservation of the $L^2$ norm, we get that
\begin{equation*}
\| \psi(t)\|^2_{L^{2 + 2\varepsilon}} \lesssim \| \psi(t) \|_{L^2}^\frac{2 - 2\varepsilon}{1 + \varepsilon} \| \psi(t) \|_{L^4}^\frac{4\varepsilon}{1 + \varepsilon} \lesssim \| \psi(t)\|^\frac{4\varepsilon}{1 + \varepsilon}_{H^\frac{1}{2}}.
\end{equation*}
Thus, for any $\varepsilon \in ( 0,1) $, it follows that
\begin{equation}
\| (1 - \Delta)^{-\frac{1}{2}} |\psi(t)|^2 \|_{L^2} \lesssim \left(\frac{2^\varepsilon}{\varepsilon}\right)^\frac{1}{p}  \| \psi(t)\|^\frac{4\varepsilon}{1 + \varepsilon}_{H^\frac{1}{2}}
\end{equation}
\end{proof}
For any $T>0$ and any $t \in [0,T]$, from \eqref{En cons}, \eqref{pot est} and \eqref{supexp est}, one has
\begin{equation*}
\langle D_m \psi(t), \psi(t) \rangle_{L^2} \lesssim {2^\frac{2\varepsilon}{p}} \varepsilon^{- \frac{2}{p}} \| \psi(t)\|^\frac{8\varepsilon}{1 + \varepsilon}_{H^\frac{1}{2}} \lesssim 2^{ \varepsilon \frac{1 + 3 \varepsilon}{1 + \varepsilon}} \varepsilon^{- \frac{1 + 3\varepsilon}{1 + \varepsilon}} e^{\frac{8\varepsilon}{1 + \varepsilon}e^{ t}} \lesssim 2^{ \varepsilon \frac{1 + 3 \varepsilon}{1 + \varepsilon}} \varepsilon^{- \frac{1 + 3\varepsilon}{1 + \varepsilon}} e^{\frac{8\varepsilon}{1 + \varepsilon}C_1e^{C_2T}},
\end{equation*}
and since for $\varepsilon>0$, $\frac{1 + 3 \varepsilon}{1 + \varepsilon} = \frac{2}{p} \in (1,3)$, it easily follows that there exists a constant $C>0$, independent of $\varepsilon$, such that
\begin{equation*}
\sup_{t \in [0,T]} \left|\langle D_m \psi(t), \psi(t) \rangle_{L^2} \right| \leq C e^\varepsilon \varepsilon^{-\frac{p}{2}}  e^{\frac{8\varepsilon}{1 + \varepsilon}C_1e^{C_2 T}}
\end{equation*}
We choose $\varepsilon = \frac{1}{C_1} e^{- C_2 T}$ with $T$ sufficiently large so that $ \varepsilon \in (0,1)$ and we get
\begin{equation*}
\sup_{t \in [0,T]} \left|\langle D_m \psi(t), \psi(t) \rangle_{L^2} \right| \leq
C e^{\frac{1}{C_1} e^{- C_2 T}} \left( C_1 e^{- C_2 T} \right)^{-\frac{p}{2}} e^{\frac{8}{1 + \varepsilon}} \lesssim e^{C_2 \frac{ p}{2}T},
\end{equation*}
where $C_2$ depends on $\|\psi_0\|_{H^\frac{1}{2}}$ and not on $T$. Since $p < 2$ we get \eqref{supexp estmm}.
\begin{remark}
In \cite{T}, it was proved that the solution to \eqref{eq:base2} scatters for for every $s > 0$, whenever the $H^s$ norm of the initial datum is small enough. In particular, as soon as the $\psi_0 \in H^\frac{1}{2}$ and there exists $\sigma \in (0 ,\frac{1}{2}]$, such that $\psi$ scatters in $H^\sigma$ (for example when $\| \psi_0 \|_{H^\sigma}$ is small enough), one can also obtain, by Sobolev's embedding, the following estimate
\begin{equation}\label{exp growth}
\| (1 - \Delta)^{-\frac{1}{2}} |\psi(t)|^2 \|_{L^2} \lesssim \| G \|_{L^p} \| \psi(t) \|_{L^{2+2\varepsilon}}^2 \lesssim C_{\sigma} \| \psi(t) \|_{H^\sigma}^2,
\end{equation}
where we have chosen $\varepsilon = \frac{\sigma}{1 - \sigma}$. This means that, eventually, in this case, the exponential growth estimate \eqref{exp growth} is not sharp and the kynetic energy will actually be bounded by a constant for all the times $t \in [0, \infty)$.
\end{remark}
\section{Appendix: Conservation of Mass}

\begin{lem} \label{l:m1}
If $\psi \in C([0,T);L^2(\R^2))$ is a solution to the equation \eqref{eq:base2} with initial datum $\psi_0$ then
for any $t \in [0,T)$ we have
\begin{equation} \label{eq.a0}
 \|\psi(t)\|_{L^2} = \|\psi_0\|_{L^2}.
\end{equation}
\end{lem}
\begin{proof}
We follow the idea from \cite{O}. For the purpose we rewrite \eqref{eq:base2} in the form
\begin{alignat}{2}\label{eq:a1}
  e^{\mathrm{i}t\mathcal{D}_m}\psi(t) = \psi_0  + \mathrm{i}\int_0^t e^{\mathrm{i}s\mathcal{D}_m} F(\psi)(s) ds,
\end{alignat}
where
$$ F(\psi)(s) = (V \ast \langle \gamma^0 \psi(s),  \psi(s) \rangle) \gamma^0\psi(s), $$
so taking the square in $L^2,$ we find
\begin{alignat}{2}\label{eq.a1a}
   \|\psi(t)\|_{L^2}^2 = \|\psi_0\|_{L^2}^2 + \left\|\int_0^t e^{\mathrm{i}s\mathcal{D}_m} F(\psi)(s) ds \right\|^2_{L^2} + 2 \mathrm{Im} \left\langle \psi_0, \int_0^t e^{\mathrm{i}s\mathcal{D}_m} F(\psi)(s) ds \right\rangle_{L^2}.
\end{alignat}
For any function $g(s) \in C([0,T];H)$ with $H$ a Hilbert space we can use the relation
$$ \left\|\int_0^t g(s) ds \right\|^2_{L^2} = 2 \mathrm{Re} \iint_{t>s > s^\prime>0} \langle g(s), g(s^\prime) \rangle_H ds ds^\prime = 2 \mathrm{Re} \int_0^t \langle g(s), \int_0^s g(s^\prime ) ds^\prime \rangle_H ds.$$
Then we can write
$$ \left\|\int_0^t e^{\mathrm{i}s\mathcal{D}_m} F(\psi)(s) ds \right\|^2_{L^2}= 2 \mathrm{Re}
 \int_0^t \langle e^{\mathrm{i}s\mathcal{D}_m} F(\psi)(s), \int_0^s e^{\mathrm{i}s^\prime\mathcal{D}_m} F(\psi)(s^\prime) ds^\prime \rangle_{L^2} ds
 $$  $$ = 2 \mathrm{Re}  \int_0^t \langle  F(\psi)(s), \int_0^s e^{-\mathrm{i}(s-s^\prime)\mathcal{D}_m} F(\psi)(s^\prime) ds^\prime \rangle_{L^2} ds $$
 and we are in position to use the integral equation
 $$  \int_0^s e^{-\mathrm{i}(s-s^\prime)\mathcal{D}_m} F(\psi)(s^\prime) ds^\prime = -\mathrm{i }\psi(s)+\mathrm{i} e^{-\mathrm{i}s\mathcal{D}_m}\psi_0  $$
 so we have
 $$\left\|\int_0^t e^{\mathrm{i}s\mathcal{D}_m} F(\psi)(s) ds \right\|^2_{L^2} = 2 \mathrm{Re}  \int_0^t \langle  F(\psi)(s), -\mathrm{i }\psi(s)\rangle_{L^2} ds +  2 \mathrm{Re}  \int_0^t \langle  F(\psi)(s),\mathrm{i} e^{-\mathrm{i}s\mathcal{D}_m} \psi_0 \rangle_{L^2} ds. $$
 Now we can take advantage of the fact that $ \langle  F(\psi)(s), \psi(s)\rangle_{L^2} $ is purely real and hence
 \begin{alignat*}{2}\label{eq.a4}
    \left\|\int_0^t e^{\mathrm{i}s\mathcal{D}_m} F(\psi)(s) ds \right\|^2_{L^2} & =  2 \mathrm{Re}  \int_0^t \langle  F(\psi)(s),\mathrm{i} e^{-\mathrm{i}s\mathcal{D}_m}f\rangle_{L^2} ds = \\ \nonumber
  & = 2 \mathrm{Im}  \int_0^t \langle e^{-\mathrm{i}s\mathcal{D}_m} F(\psi)(s), f\rangle_{L^2} ds =
  - 2 \mathrm{Im}  \int_0^t \langle f, e^{-\mathrm{i}s\mathcal{D}_m} F(\psi)(s)\rangle_{L^2} ds
 \end{alignat*}
 and from \eqref{eq.a1a} we arrive at the desired identity \eqref{eq.a0}.
\end{proof}

\subsection{Acknowledgements}

The first  author was supported in part by  INDAM, GNAMPA - Gruppo Nazionale per l'Analisi Matematica, la Probabilit\`a e le loro Applicazioni, by Institute of Mathematics and Informatics, Bulgarian Academy of Sciences, by Top Global University Project, Waseda University,  the Project PRA 2018 49 of  University of Pisa and project "Dinamica di equazioni nonlineari dispersive", "Fondazione di Sardegna", 2016.

\end{document}